\documentclass[12pt]{article}
\usepackage[margin=1in]{geometry}
\usepackage{graphicx}
\usepackage[hidelinks]{hyperref}
\usepackage{subcaption}
\usepackage{float}
\usepackage{amssymb,amsmath,amsthm}
\usepackage{newtxtext,newtxmath}
\usepackage[dvipsnames]{xcolor}
\usepackage{tikz-cd}
\usetikzlibrary{decorations.markings}

\usepackage[backend=biber,style=numeric,maxbibnames=8]{biblatex}
\definecolor{addblue}{rgb}{0.1,0,0.8}

\usepackage[normalem]{ulem}

\definecolor{darkgrn}{rgb}{0,0.75,0}

\theoremstyle{plain} 
\newtheorem{theorem}{Theorem}[section]
\newtheorem{lemma}[theorem]{Lemma}
\newtheorem{proposition}[theorem]{Proposition}

\theoremstyle{definition} 
\newtheorem{definition}[theorem]{Definition}

\theoremstyle{remark} 
\newtheorem*{remark}{Remark}

\newcommand{\R}{{\mathbb R}}

\newcommand{\Q}{{\mathbb Q}}

\newcommand{\N}{{\mathbb N}}

\title{\huge Existence of Minimal Homotopies for Immersed Planar Curves}

\author{Lia Buchbinder\thanks{liya.boukhbinder@wsu.edu; corresponding author} \quad Yunjia Kou \quad Bala Krishnamoorthy \quad Kevin R.~Vixie\\
Washington State University}
\date{} 

\begin{document}
\maketitle

\begin{abstract}
We study the existence of area-minimizing homotopies between homotopic curves in the plane. 
While the classical Plateau problem establishes the
existence of least-area surfaces spanning a single Jordan curve, the corresponding existence theory for homotopies between curves is more subtle and is not directly covered by the same framework. 
Existing results in the plane are mainly based on combinatorial and algebraic methods, such as decomposing curves into self-overlapping subcurves. 
These methods are highly effective in the planar setting, but they are often tied to special classes of curves and rely strongly on the local structure of the self-intersections, frequently assuming transverse crossings. 
In contrast, our approach is geometric and variational, and does not depend on the local structure of the self-intersections.

In this paper, we develop a variational existence theory for minimum-area homotopies of immersed planar curves. 
Our approach adapts classical minimal surface methods by lifting an immersed planar curve with self-intersections into higher co-dimension, where it becomes embedded. 
For such a lifted curve, we apply Douglas’s solution of the Plateau problem to obtain an area-minimizing disk. 
For closed curves of class $C^1$, we prove uniform convergence of the Douglas minimizers and show that the limiting map minimizes area among all $C^1$ spanning maps of the original planar curve. 
We then extend the construction to closed Lipschitz curves using approximation and Sobolev compactness arguments.
Since the limiting minimizing disk lies in the original plane, it directly produces a null homotopy whose swept area is minimal among all admissible homotopies of the original curve. 
In this way, the construction connects Plateau theory with
minimal homotopy area minimization.
\end{abstract}

\section{Introduction} \label{sec:intro}
Given two open homotopic curves in the plane, a natural geometric problem is to understand whether there exists a homotopy between them that minimizes the total area swept during the deformation.
Such area-minimizing homotopies arise naturally in topology and geometric measure theory, and play an important role in applications ranging from curve similarity and shape matching to geometric optimization.
While Plateau’s problem asks whether there exists a surface of least area whose boundary coincides with a given closed, embedded curve, the corresponding existence problem for area-minimizing homotopies between curves is not resolved by the framework for Plateau's framework.

\subsection{Background} \label{subsec:backgrnd}
The existence problem for area-minimizing objects associated with curves has a long history, beginning with Plateau’s problem, which asks whether a least-area surface exists whose boundary coincides with a given closed curve in Euclidean space.
Douglas \cite{Do1931} and Radó \cite{Ra1930} independently proved that for any Jordan curve in $\R^{n}$ there exists a surface of least area spanning the curve.
Their work establishes existence of minimal disks parametrized by harmonic maps and provides the foundational analytic framework for minimal surface theory.

The problem of minimizing swept area under a homotopy between curves, however, is of a different nature, and existence results for such homotopies are not a direct consequence of the Douglas-Radó results.
For planar curves, existence results are known under additional structural assumptions. Early combinatorial ideas already appear in Blank work \cite{Sa1967}, who introduced algebraic word-based encodings of planar curves to study area-minimizing homotopies in a purely combinatorial framework.
Chambers and Wang \cite{ChWa2013} introduced minimum homotopy area as a measure of curve similarity for simple homotopic curves, focusing on its geometric properties and computation in settings where minimizing homotopies are assumed to exist.
An existence result was obtained by Nie \cite{Ni2014}, who studied minimum-area null-homotopies and homotopies between closed plane curves using a combinatorial and algebraic framework.
Assuming that the curves divide the plane into finitely many regions, he proved that a minimum-area homotopy exists within the class of piecewise smooth homotopies.
The proof relies on an algebraic encoding of planar curves rather than on variational compactness or analytic minimal surface methods.
Fasy, Karako\c{c}, and Wenk \cite{FaKaWe2017}, studied minimum null-homotopies of normal curves using a framework based on homotopy moves. 
They showed that a minimum null-homotopy can be constructed by decomposing the curve into a sequence of self-overlapping subcurves.
Each such subcurve bounds an immersed disk and can be contracted in turn, yielding a minimum-area homotopy for the original curve.
Evans and Wenk \cite{EvWe2023} investigated the relationship between minimum homotopy area and self-overlapping curves. 
They provided combinatorial characterizations of when a curve is self-overlapping and show how minimum homotopy area can be understood in terms of winding area and interior boundary structure. 
Chang, Fasy, McCoy, Millman, and Wenk \cite{ChFaMcMiWe2023} developed a combinatorial approach for computing minimum-area null-homotopy of normal curves. 
Encoding the curve as a word via Blank and Nie constructions, they used dynamic programming to compute the minimum homotopy area directly from the combinatorial representation.
All of the above approaches rely, in one form or another, on additional combinatorial or structural assumptions on the curves, such as transverse self-intersections, normality, or the existence of suitable self-overlapping decompositions. 

From a geometric measure theory perspective, existence of area-minimizing fillings is guaranteed in the sense of integral currents \cite{Fe1969}.
However, such minimizers may fail to correspond to a minimal continuous deformation of curves.
Currents allow cancellation through multiplicity and orientation.
For a simple illustration, consider two homotopic open curves with the same endpoints, where one curve traverses a planar region in one direction and the other traverses
the same region in the opposite direction. 
Concatenating the two curves produces a figure-eight type closed curve whose two lobes have opposite orientations.
If the two lobes bound the same planar region with opposite orientation, then as an integral current they cancel, and the mass of the minimizing current is zero 
In contrast, any continuous homotopy must sweep both lobes, and its area necessarily includes the area of each loop. Thus, current-theoretic minimizers notion of minimality does not always coincide with minimal homotopy area (See Figure \ref{fig:figure8_curves.png}). 

\begin{figure} [ht]
        \centering
    \includegraphics[width=0.75\linewidth]{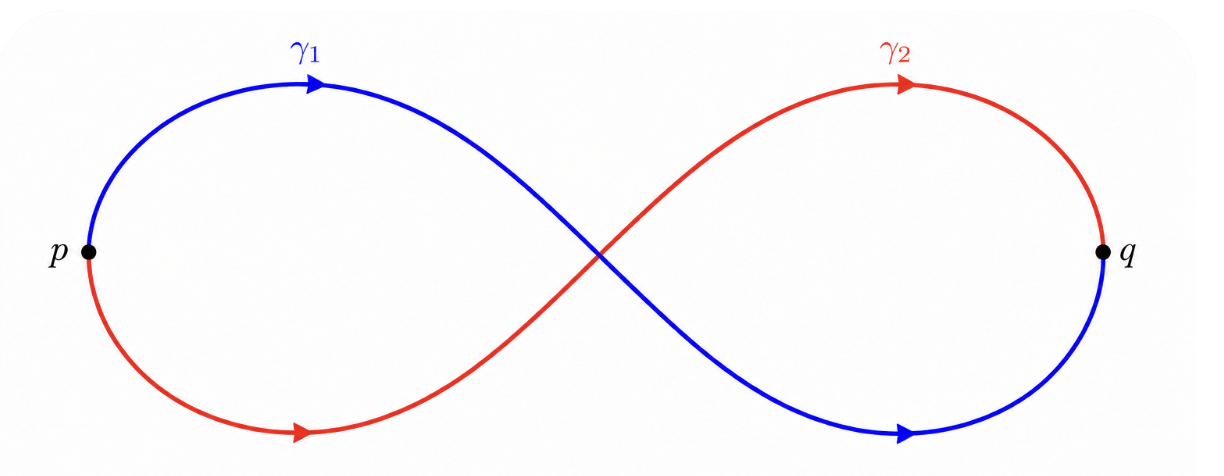}
    \caption{Two homotopic curves $\gamma_1$ (blue) and $\gamma_2$ (red) with the same endpoints $p$ and $q$. 
Their concatenation $\gamma=\gamma_1 * \gamma_2^{-1}$ forms a closed figure-eight curve whose two lobes have opposite orientations.}
    \label{fig:figure8_curves.png}
\end{figure}

White \cite{Wh1983} proves existence results for least-area mappings under embedding assumptions.
But a closed curve arising from two homotopic curves can have self-intersections.
As a result, existence theorems formulated for embedded boundaries do not directly apply to the homotopy area minimization problem.

\subsection{Our Contributions} \label{subsec:contrb}
The purpose of this paper is to provide a variational existence theory for minimum-area homotopies between immersed planar curves as described in Section~\ref{sec:settings}.
Our approach is geometric and variational, producing minimizing homotopies as limits of area-minimizing maps arising from Plateau-type problems. 
In particular, the method is independent of the local structure of the self-intersections and does not require the curve to be normal or to have transverse crossings.

We introduce a lifting construction that embeds an immersed planar curve into $\R^{4}$ by perturbing it in an orthogonal two-dimensional direction.
This separates self-intersections while remaining uniformly close to the original curve. 
Applying Douglas’s solution of the Plateau problem to the lifted curves produces a family of area-minimizing disks.
Using the classical three-point normalization together with the Courant-Lebesgue equicontinuity lemma, we prove uniform equicontinuity of the boundary parametrizations and obtain compactness of the corresponding
Douglas minimizers (Theorem~\ref{theo:unfrm-cnvg-Dmap}). 
This allows us to construct a limiting minimal disk spanning the original planar curve .
For closed curves of class $C^1$, we show in Proposition~\ref{prop:min-area} that the limiting map minimizes area among all $C^1$ spanning maps of the original planar curve.
Since the limiting map remains in $\R^2$, it directly gives a null homotopy whose swept area is minimal among all Lipschitz null homotopies of the original curve.
We then extend the construction to closed Lipschitz curves using approximation and Sobolev compactness arguments. 
In this setting, we prove weak convergence in $W^{1,2}$, strong convergence in $L^2$, identification of the boundary trace through a monotone parametrization, and continuity of the limiting map up to the boundary (see Theorem~\ref{theo:conv} and Proposition~\ref{prop:Minimality-of-the-Limit Surface}).

Conceptually, this provides a variational bridge between Plateau theory and homotopy area minimization.
The notation used throughout the paper is summarized in
Table \ref{tab:notation}.

\begin{table}[ht]
\centering
\caption{Notation used throughout the paper.}
\label{tab:notation}
\begin{tabular}{@{}p{0.22\textwidth} p{0.72\textwidth}@{}}
\hline
\textbf{Notation} & \textbf{Description} \\
\hline
$\R^{n}$ & $n$-dimensional Euclidean space with the standard metric \\

$\mathcal{H}^{k}$  & $k$-dimensional Hausdorff measure \\

$S^{1}$ & Unit circle in $\R^{2}$ with its standard orientation \\

$D^{2}$  & Unit disk in $\R^{2}$ \\

$\gamma : S^{1} \to \R^{2}$  & Closed immersed planar curve \\

$\gamma_\varepsilon \subset \R^{4}$   & Lifted curve obtained from $\gamma$ with parameter $\varepsilon>0$ \\

$\gamma_{0} \subset \R^{4}$  & Flattened curve $(\gamma,0,0)$ in $\R^{4}$ \\
    
$Lip(f)$  & Lipschitz constant of a map $f$ \\

$W^{1,2}(D^{2})$  & Sobolev space of maps with square-integrable weak derivatives \\

$H^{1/2}(\partial D^{2})$  & Fractional Sobolev trace space on $\partial D^{2}$ \\

$u_\varepsilon : D^{2} \to \R^{4}$ & Douglas area-minimizing disk spanning $\gamma_\varepsilon$ \\

$u_{0} : D^{2} \to \R^{4}$  & Uniform limit of $\{u_\varepsilon\}$ as $\varepsilon \to 0$ \\

$g_\varepsilon^{*}$  & Douglas boundary parametrization of $\gamma_\varepsilon$ \\

$\phi_\varepsilon$  & Monotone boundary reparametrization associated with $u_\varepsilon$ \\

$\phi_0$  & Limiting monotone boundary parametrization \\

$P:\R^{4}\to\R^{2}$  & Orthogonal projection onto the first $\R^{2}$ factor \\

$\widetilde{\mathscr{h}}$  & Lifted minimal homotopy before projection to $\R^{2}$ \\

$\mathscr{h}_0$  & Projected minimal homotopy map in $\R^{2}$ \\

$\mathscr{E}(u)$  & Dirichlet energy of a map $u$ \\

$J_u$  & Jacobian of a map $u$ \\

$\mathrm{Area}(u)$ & Parametric area of a map defined by the area formula \\

\hline
\end{tabular}
\end{table}

\section{Settings and Definitions} \label{sec:settings}

In this section, we describe definitions and assumptions
used throughout the paper.
We begin with homotopies of curves in the plane. 
\begin{definition}
Let $\gamma_1, \gamma_2 : [a,b] \to \R^2$ be two open curves that share the same endpoints. 
A \textit{homotopy} between them is a continuous map
$$
H : [0,1] \times [a,b] \to \R^2
$$
such that
$$
H(0,t)=\gamma_1(t), \qquad H(1,t)=\gamma_2(t),
$$
and the endpoints remain fixed throughout the deformation. 
If, in addition, each intermediate curve $H(s,\cdot)$ is injective, the homotopy is called an isotopy.
\end{definition}

For closed curves, we view them as continuous maps $\gamma : S^1 \to \R^2$. 

\begin{definition}\label{def:nullhomo}
A \textit{null homotopy} of a closed curve $\gamma$ is a homotopy between $\gamma$ and a constant curve. 
\end{definition}
Since $R^2$ is simply connected, every closed curve admits a null homotopy.

\begin{definition}
A map $f : S^1 \to \R^n$ is \textit{Lipschitz} if there exists $L \ge 0$ such that
$$
\|f(t_1) - f(t_2)\| \le L |t_1 - t_2|
\quad \text{for all } t_1,t_2,
$$
where $\|\cdot\|$ denotes the Euclidean norm in $\R^n$ and $|\cdot|$ denotes the Euclidean distance in the parameter domain.
\end{definition}
Every $C^1$ curve on the compact domain $S^1$ is automatically Lipschitz.

In this paper, we prove the existence of minimal homotopies minimizing area over Lipschitz competitors. 
This is the natural setting for our approach since the area formula and the energies applies directly to Lipschitz maps.

Also, We believe that an alternative formulation based on minimizing the two-dimensional Hausdorff measure of the continuous image of the disk may also produce minimizers with Lipschitz parameterizations. 

Next, we introduce the \textit{Sobolev space} $W^{1,2}$. 
\begin{definition}
The Sobolev space $W^{1,2}(U)$\cite{HeKoShTy2015} 
is the normed space of maps
$u \in L^2(U)$
whose weak gradient $\nabla u$ belongs to
$L^2(U)$, equipped with the norm
$$
\|u\|_{W^{1,2}(U)}
:=
\|u\|_{L^2(U)}
+
\|\nabla u\|_{L^2(U)}.
$$
\end{definition}

In what follows, we consider two open,  homotopic, $C^1$ curves 
$$
\gamma_1, \gamma_2 : [a,b] \to \R^2
$$
that share the same endpoints and whose first derivatives agree at those endpoints (after reversing the orientation of $\gamma_2$). 
Then the concatenation
$$
\gamma\equiv\gamma_1\cup\overline{\gamma_2}
$$
is a closed curve of class $C^1$; in particular, it defines an immersion of $S^1$.
We study null homotopies of this curve. 

\paragraph{The lifting construction.} \label{par:lift}
Now, we describe the geometric construction that allows us to remove self-intersections of $\gamma$ by lifting it into $\R^4$.
We work in ambient space $\R^2 \times \R^2 \cong \R^4$ and identify
$$
S^1 = \{ (\cos t, \sin t) : t \in [0,2\pi) \}.
$$
For $\varepsilon > 0$ we define the lifted curve
$$
\gamma_\varepsilon(t)
=
\big(\gamma(t), \varepsilon \cos t, \varepsilon \sin t \big)
\subset \R^4 .
$$
Thus, the original curve lies in the first $\R^2$, and the additional two coordinates provide a small orthogonal displacement in the second $\R^2$.
If $\gamma(t_1)=\gamma(t_2)$ with $t_1 \neq t_2$, then the last two coordinates satisfy
$(\cos t_1, \sin t_1) \neq (\cos t_2, \sin t_2)$,
so $\gamma_\varepsilon(t_1) \neq \gamma_\varepsilon(t_2)$.
Hence, for every $\varepsilon>0$, all self-intersections of $\gamma$ are separated in $\R^4$, and no new intersections are created.
Consequently, for every $\varepsilon>0$, the lifted curve $\gamma_\varepsilon$ is embedded in $\R^4$ (See Figure~\ref{fig:lift} for an illustration).

\begin{figure} [ht]
    \centering
    \includegraphics[width=0.95\linewidth]{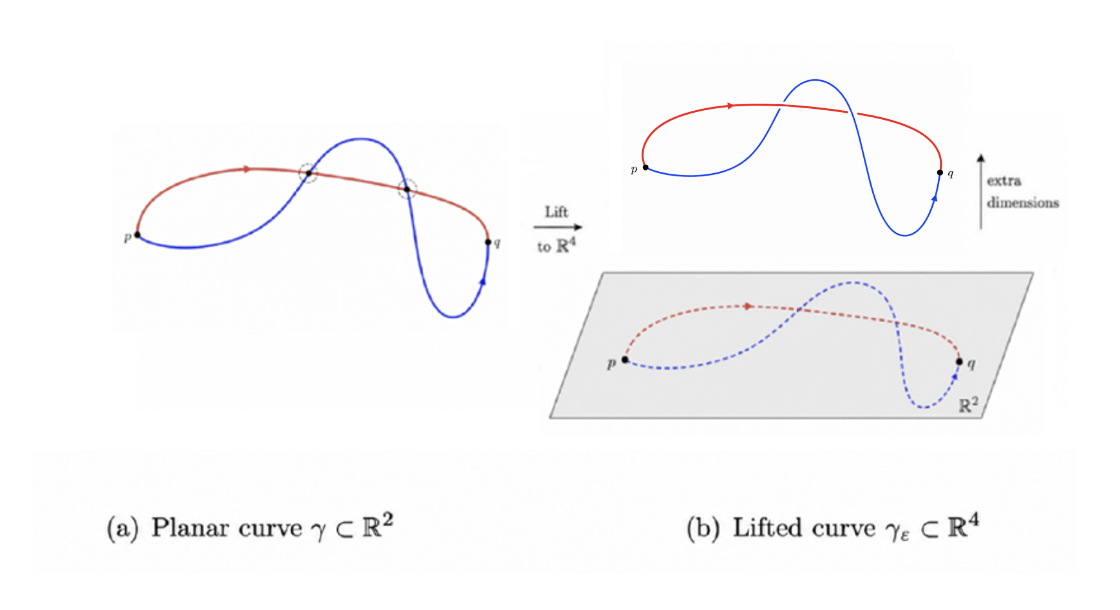}
    \caption{(a) A planar immersed curve $\gamma$ with self-intersections. 
    The circled regions indicate neighborhoods where the strands intersect in the planar projection. 
    (b) A lifted curve $\gamma_{\varepsilon}$ obtained by separating the intersecting in the additional dimensions, producing an embedded curve in $R^4$.
}
    \label{fig:lift}
\end{figure}

\paragraph{The three-point normalization.} \label{par:3pt}
We also use the \textit{three-point normalization} \cite{Pe2025} for the curves $\gamma_\varepsilon$. 
Boundary parametrizations of Douglas minimizers are not unique, since composing with a Möbius transformation of the disk changes the parametrization while preserving both the geometric image of the curve and the Dirichlet energy. 
Thus, without a normalization, the boundary maps may drift or degenerate.

Given $z_1,z_2,z_3\in\partial D^2$ and $w_1,w_2,w_3\in\partial D^2$ two triples in $\partial D^2$ taken with the same order, there exists a unique M\"obius transformation $\varphi:D^2\to D^2$ such that $\varphi(z_k)=w_k$ for $k=1,2,3$. 
Now, for any fixed $\varepsilon$, we select three distinct points $p_1,p_2,p_3\in \gamma_\varepsilon(S^1)$. 
Since the Dirichlet energy $\mathscr{E}$ satisfying $\mathscr{E}(f\circ\varphi)=\mathscr{E}(f)$ for $\varphi$ M\"obius, we can assume that each curve $\gamma_\varepsilon$ of the minimizing sequence satisfies
$$
\gamma_\varepsilon(e^{\frac{2k\pi}{3}i})=p_k,\ \ k=1,2,3.
$$
This normalization does not change the minimizing property, harmonicity, conformality, or energy of $u_\varepsilon$. 
It only stabilizes the way the boundary curve is parametrized.

\section{Main Results}
In this section, we present the main existence results of the paper.
We first establish the argument in the $C^1$ setting and then extend
the result to Lipschitz curves.

Before stating the main compactness result, we recall the classical existence result for Plateau's problem. 
If $\gamma \subset \R^n$ is a Jordan curve, then Douglas theorem there exists a harmonic and conformal map $u : D^2 \to \R^n$ that minimizes area among all maps spanning $\gamma$ and extends continuously to $\overline{D^2}$ with a parametrization $u|_{\partial D^2}$ of $\gamma$ \cite{Do1931}.

Since for each $\varepsilon>0$ the lifted curve $\gamma_\varepsilon \subset \R^4$ is a Jordan curve, the theorem applies to $\gamma_\varepsilon$. 
We therefore obtain a minimizing disk
$u_\varepsilon : D^2 \to \R^4$ spanning $\gamma_\varepsilon$.
We now study the behavior of this family as $\varepsilon \to 0$.

\subsection{Existence in the \texorpdfstring{$C^1$}{C1} Setting} \label{subsec:C1}
In this section, we will show the existence of the minimizer among all the Lipschitz competitors. We first prove a compactness result for the Douglas minimizers associated to the lifted curves. 
This will allow us to pass to a limiting disk as the
lift parameter tends to zero, which will later be used to construct a minimal null homotopy for the original curve.

\begin{theorem} \label{theo:unfrm-cnvg-Dmap}
Let $\gamma:S^{1}\to\R^{2}$ be a closed $C^1$ curve. 
Let $\gamma_{\varepsilon}\subset\R^{4}$ be the lifted Jordan curves defined in Section~2.1, with $\gamma_{0}(t)=(\gamma(t),0,0)$.
For each $0<\varepsilon\le1$, let $u_{\varepsilon}:D^{2}\to\R^4$ denote
a Douglas minimizer spanning $\gamma_{\varepsilon}$,
where $g_{\varepsilon}^{*}:S^{1}\to\gamma_{\varepsilon}$ denote the Douglas parametrizations.
Then there exists a sequence $\varepsilon_{j}\downarrow0$ and a continuous map $u_{0}:\overline{D^{2}}\to\mathbb{R}^{4}$ such that
\[
u_{\varepsilon_{j}} \longrightarrow u_{0}
\quad\text{uniformly on }\overline{D^{2}}.
\]
Moreover, the boundary trace of $u_{0}$ coincides with $g_{0}^{*}$, i.e.
\[
u_{0}\big|_{\partial D^{2}}=g_{0}^{*}.
\]
\end{theorem} 

\begin{proof}
\textbf{Step 1: Uniform bound for $\{u_{\varepsilon}\}$.} 
Since $\gamma$ is continuous and the functions $\cos t$ and $\sin t$ are smooth, it follows that each lifted curve
$$
\gamma_\varepsilon:S^1\to\R^4
$$
is continuous.
Since $S^1$ is compact, $\gamma$ is bounded. Hence there exists a constant $M>0$ such that
$$
|\gamma(t)|\le M
\qquad\text{for all }t\in S^1.
$$
Recall that the lifted curves are defined by
$$
\gamma_\varepsilon(t)
=
(\gamma(t),\varepsilon\cos t,\varepsilon\sin t),
\qquad 0<\varepsilon\le1.
$$
Therefore,
$$
|\gamma_\varepsilon(t)|
\le
|\gamma(t)|+\varepsilon
\le
M+1,
$$
so
$$
\sup_{0<\varepsilon\le1}
\|\gamma_\varepsilon\|_{L^\infty(S^1)}
\le
M+1.
$$
Since the Douglas boundary parametrization has the form
$$
g_\varepsilon^*=\gamma_\varepsilon\circ\phi_\varepsilon,
$$
where $\phi_\varepsilon:S^1\to S^1$ is weakly monotone and continuous, it follows that $g_\varepsilon^*$ is continuous on $S^1$.
Also,
$$
g_\varepsilon^*(S^1)=\gamma_\varepsilon(S^1),
$$
the two maps have the same image. Consequently,
$$
\|g_\varepsilon^*\|_{L^\infty(S^1)}
=
\|\gamma_\varepsilon\|_{L^\infty(S^1)}.
$$
Hence,
$$
\sup_{0<\varepsilon\le1}
\|g_\varepsilon^*\|_{L^\infty(S^1)}
\le
M+1.
$$
Each coordinate of $u_\varepsilon$ is harmonic in $D^2$ and continuous on $\overline{D^2}$ with boundary values $g_\varepsilon^*$.
By the maximum principle \cite{Ev2022},
$$
\|u_\varepsilon\|_{L^\infty(\overline{D^2})}
\le
\|g_\varepsilon^*\|_{L^\infty(S^1)}.
$$
Therefore,
$$
\sup_{0<\varepsilon\le 1}\|u_\varepsilon\|_{L^\infty(\overline{D^2})}\le C.
$$

\noindent
\textbf{Step 2: Uniform convergence on the boundary.}
We first show that the family $\{u_\varepsilon\}_{0<\varepsilon\le1}$
has uniformly bounded Dirichlet energy.

For each $0<\varepsilon\le1$, let
$$
v_\varepsilon(re^{it})=r\gamma_\varepsilon(t),
\qquad 0\le r\le1,
$$
be the cone competitor spanning $\gamma_\varepsilon$.
Since $u_\varepsilon$ is a Douglas minimizer spanning $\gamma_\varepsilon$,
it minimizes the area energy, as well as Dirichlet energy, among all competitors with the same boundary values\cite{Pe2025}. 
Hence,
$$
\mathscr{E}(u_\varepsilon)\le \mathscr{E}(v_\varepsilon).
$$
Now,
$$
(v_\varepsilon)_r=\gamma_\varepsilon(t),
\qquad
(v_\varepsilon)_t=r\gamma_\varepsilon'(t).
$$
Therefore,
$$
|(v_\varepsilon)_r|^2+|(v_\varepsilon)_t|^2
=
|\gamma_\varepsilon(t)|^2+r^2|\gamma_\varepsilon'(t)|^2.
$$
Since
$$
\gamma_\varepsilon(t)
=
(\gamma(t),\varepsilon\cos t,\varepsilon\sin t),
$$
we have
$$
|\gamma_\varepsilon(t)|
\le
|\gamma(t)|+1,
$$
and
$$
|\gamma_\varepsilon'(t)|^2
=
|\gamma'(t)|^2+\varepsilon^2.
$$
Because $\gamma\in C^1(S^1)$ and $S^1$ is compact, both $\gamma$ and $\gamma'$ are bounded. Hence there exists a constant $C>0$,
independent of $\varepsilon$, such that
$$
\mathscr{E}(v_\varepsilon)\le C.
$$
Consequently,
$$
\mathscr{E}(u_\varepsilon)\le \mathscr{E}(v_\varepsilon)\le C,
$$
for all $0<\varepsilon\le1$.

By the previous energy estimate, there exists a constant $C>0$, independent of
$\varepsilon$, such that
$
\mathscr{E}(u_\varepsilon)\le C
$
for all $0<\varepsilon\le 1$.

Recall that a family of functions $\mathcal{F}$ on a metric space is called equicontinuous if for every $\eta>0$ there exists $\delta>0$ such that
\[
|x-y|<\delta \quad \Rightarrow \quad |f(x)-f(y)|<\eta
\]
for all $f\in\mathcal{F}$ \cite{Ru1976}. 

As we explained in section~\ref{sec:settings}, we may assume that the family $\{u_\varepsilon\}$ satisfies the three-point condition. 
That is, for three fixed distinct points $e^{i\theta_1}, e^{i\theta_2}, e^{i\theta_3} \in \partial D^2$ and three fixed distinct points $p_1, p_2, p_3 \in S^1$ such that $\gamma(p_1), \gamma(p_2), \gamma(p_3)$ are also distinct, it is valid that 
$$g_\epsilon^*(e^{i\theta_k})=\gamma_\epsilon(p_k),\ \ k=1,2,3$$ for every $\epsilon$. From the Courant-Lebesgue Lemma \cite{Pe2025},  
$$
|z-w|<\delta
\quad\Longrightarrow\quad
|g_\varepsilon^*(z)-g_\varepsilon^*(w)|<\eta
$$
on $\partial D^2$. 
Therefore, the family $\{g_\varepsilon^*\}_{0<\varepsilon\le1}$ is
equicontinuous on $\partial D^2$.

Since $\partial D^2$ is compact and the family $\{g_\varepsilon^*\}$ is uniformly bounded and equicontinuous, the Arzelà–Ascoli theorem \cite{Ru1976} implies that there exists a subsequence $\varepsilon_j \downarrow 0$ and a
continuous map
$$
g_0 : \partial D^2 \to \R^4
$$
such that
$$
g_{\varepsilon_j}^* \to g_0^*
\quad \text{uniformly on } \partial D^2.
$$

\noindent
\textbf{Step 3: Uniform convergence of $\{u_\varepsilon\}$}
Each coordinate of $u_\varepsilon$ is harmonic in $D^2$ and continuous on $\overline{D^2}$ with boundary values $\gamma_\varepsilon$.
By the maximum principle \cite{Ev2022},
$$
\|u_{\varepsilon_j}-u_{\varepsilon_k}\|_{L^\infty(\overline{D^2})}
\le
\|g_{\varepsilon_j}^*-g_{\varepsilon_k}^*\|_{L^\infty(\partial D^2)}.
$$
Since
$
g_{\varepsilon_j}^*\to g_0^*
$
uniformly on $\partial D^2$, the sequence
$
\{g_{\varepsilon_j}^*\}
$
is Cauchy in $L^\infty(\partial D^2)$.
Therefore,
$
\{u_{\varepsilon_j}\}
$
is Cauchy in $L^\infty(\overline{D^2})$.
Since $C^0(\overline{D^2})$ is complete, there exists a continuous map
$
u_0:\overline{D^2}\to\mathbb R^4
$
such that
$$
u_{\varepsilon_j}\to u_0
\quad\text{uniformly on }\overline{D^2}.
$$

Since each $u_{\varepsilon_j}$ is continuous on $\overline{D^2}$ and
$
u_{\varepsilon_j}\to u_0
$
uniformly on $\overline{D^2}$, the limit map $u_0$ is continuous on $\overline{D^2}$.

\noindent
\textbf{Step 4: identification of the boundary.}
For each $j$,
$
u_{\varepsilon_j}|_{\partial D^2}=g_{\varepsilon_j}^*.
$
Since
$
u_{\varepsilon_j}\to u_0
$
uniformly on $\overline{D^2}$, restricting to the boundary gives
$
u_{\varepsilon_j}|_{\partial D^2}\to u_0|_{\partial D^2}
$
uniformly on $\partial D^2$. Hence,
$
g_{\varepsilon_j}^*\to u_0|_{\partial D^2}
$
uniformly on $\partial D^2$.
On the other hand, from Step 2 we already know that
$
g_{\varepsilon_j}^*\to g_0^*
$
uniformly on $\partial D^2$.
By uniqueness of uniform limits,
\[
u_0|_{\partial D^2}=g_0^*.
\]
\end{proof}

Next, we prove that the limit map $u_0$ minimizes area among all Lipschitz maps spanning $g_0^{*}$.

\begin{proposition} \label{prop:min-area}
Let $v : \overline{D^2} \to \R^4$ be a Lipschitz map with $v|_{\partial D^2} = g_0^{*}$.
We define area by the area formula \cite{Fe1969}
\[
\operatorname{Area}(v) = \int_{D^2} J_v(x)\, d \mathcal{H}^2(x),
\qquad
J_v = \sqrt{\det(Dv^T Dv)}.
\]
The limit map $u_0$ from Theorem~\ref{theo:unfrm-cnvg-Dmap} minimizes area among all Lipschitz maps
$v:\overline{D^2}\to\R^4$
with
$
v|_{\partial D^2}=g_0^*.
$
\end{proposition}

\begin{proof} \textbf{Step 1: Boundary adjustment.}
Fix $\delta \in (0,1/2)$ and write $x = re^{it}$.
Define $v_{\varepsilon,\delta}$ by
\[
v_{\varepsilon,\delta}(re^{it}) =
\begin{cases}
v(re^{it}), & 0 \le r \le 1-\delta, \\[2mm]
(1-s)\, v((1-\delta)e^{it}) + s\, \gamma_\varepsilon(t),
& 1-\delta \le r \le 1,
\end{cases}
\qquad
s = \dfrac{r-(1-\delta)}{\delta}.
\]
Then
\[
v_{\varepsilon,\delta}|_{\partial D^{2}}=g_\varepsilon^{*},
\qquad
v_{\varepsilon,\delta}=v \quad \text{on } |x|\le1-\delta.
\]

Since $g_\varepsilon^{*}\to g_0^{*}$ uniformly on $S^{1}$ and $v$ is continuous on $\overline{D^{2}}$, choosing $\delta=\sqrt{\varepsilon}$ gives
\[
\operatorname{Area}(v_{\varepsilon,\delta})
=
\operatorname{Area}(v)+o(1)
\qquad (\varepsilon\to0).
\]

For simplicity, we write
\[
v_\varepsilon=v_{\varepsilon,\delta(\varepsilon)}.
\]

\noindent
\paragraph{Step 2: Comparison.}
Since $u_\varepsilon$ minimizes area among maps spanning
$g_0*$,
\[
\operatorname{Area}(u_\varepsilon)
\le
\operatorname{Area}(v_\varepsilon).
\]
Hence,
\[
\limsup_{\varepsilon\to0} \operatorname{Area}(u_\varepsilon)
\le
\limsup_{\varepsilon\to0} \operatorname{Area}(v_\varepsilon)
\le
\operatorname{Area}(v).
\]

\noindent
\paragraph{Step 3: Continuity of area.}
Since $\{u_\varepsilon\}$ is uniformly bounded in $C^1(\overline{D^2})$,
after passing to a subsequence from Theorem~\ref{theo:unfrm-cnvg-Dmap}, we have
$u_{\varepsilon_j} \to u_0$ in $C^{1}(\overline{D^2})$.
Hence $Du_{\varepsilon_j} \to Du_0$ uniformly.
Since $A \mapsto \sqrt{\det(A^T A)}$ is continuous,
\[
\operatorname{Area}(u_{\varepsilon_j})
\to
\operatorname{Area}(u_0).
\]
In particular,
\[
Area(u_0)
=
\lim_{j\to\infty} Area(u_{\varepsilon_j})
\le
\liminf_{\varepsilon\to0} Area(u_\varepsilon).
\]

\noindent
\paragraph{Step 4: Conclusion.}
Combining the inequalities,
\[
\operatorname{Area}(u_0)
\le
\operatorname{Area}(v).
\]
Since $v$ was arbitrary, $u_0$ minimizes area among all maps
spanning $g_0^{*}$.
\end{proof}

Let $\gamma:S^1\to\R^2$ be our original curve and view $\gamma_0(t)=(\gamma(t),0,0)\in\R^4$.
Let $u_0:\overline{D^{2}}\to\R^{4}$ be the area minimizer that spans $g_0^*=u_0|_{\partial D^2}$, where $\phi_0:S^1\to S^1$ is weakly monotone map such that $g_0^*=\gamma_0\circ\phi_0$. 

Since
\[
u_0|_{\partial D^2}=g_0^*=(\gamma\circ\phi_0,0,0),
\]
the third and fourth coordinate functions of $u_0$ are harmonic on $D^2$
with zero boundary values. By the maximum principle, they vanish identically.
Hence,
\[
u_0(D^2)\subset \mathbb{R}^2\times\{(0,0)\}.
\]
Therefore, identifying $\mathbb{R}^2\times\{(0,0)\}$ with $\mathbb{R}^2$, the map
$u_0$ itself minimizes area among all Lipschitz maps spanning $\gamma\circ\phi_0$.

\begin{remark}
The minimization is carried out among competitors
whose boundary parametrization agrees with $g_0^*$. 
This does not affect the homotopy area problem, since reparametrizing the boundary curve changes only the parametrization of the homotopy and not its geometric image or swept area. 
In the Subsection~\ref{subsec:constrct-minhom} we show how to recover a minimizing homotopy for the original curve $\gamma$.
\end{remark}

\subsection{Generalization to Lipschitz curves} \label{subsec:gnrlztion}
In this subsection, we extend the previous result in Theorem~\ref{theo:unfrm-cnvg-Dmap} to arbitrary closed Lipschitz curves.
Although the argument becomes more intricate, this additional complexity allows us to establish the existence of a minimal homotopy under weaker regularity assumptions.

We work in the Sobolev space $W^{1,2}(D^2,\R^4)$ and study limits of Douglas minimizers associated with the lifted curves $\gamma_\varepsilon$.


Before proving the main theorem of this section, we first establish the following Lemma for monotone boundary parametrizations.

\begin{lemma} \label{lem:mono}
For a family of monotone function $\{\phi_\varepsilon:[0,2\pi)\to [0,2\pi]\}$, there exists a subsequence $\{\phi_{\varepsilon_j}\}$ and a monotone function $\phi_0$ such that
$$
\phi_{\varepsilon_j}\xrightarrow{j\to\infty}\phi_0\ \ \text{pointwise}.
$$
\end{lemma}

\begin{proof} \textbf{Step 1: Extracting a subsequence on the rational points}. 
WLOG, we assume that $\phi_\varepsilon$ is increasing for all $\varepsilon$. 
Let $\{q_k\}_{k=1}^\infty$ be an enumeration of the rational numbers $\Q\cap[0,2\pi)$. 
Since $\{\phi_{\varepsilon}$\} is uniformly bounded, the sequence $\{\phi_\varepsilon(q_1)\}$ is bounded. 

By the Bolzano-Weierstrass theorem \cite{Ru1976}, we can select a convergent subsequence, denoted by $\{\phi_{\varepsilon^{(1)}_n}\}$ (with indices $\varepsilon^{(1)}_n\in N_1$).
For $q_2$, we can extract a further convergent subsequence $\{\phi_{\varepsilon^{(2)}_n}\}$ (with indices in $\varepsilon^{(2)}_n\in N_2 \subset N_1$). 
Repeating this procedure yields a nested sequence of index sets $N_1 \supset N_2 \supset \cdots$.

Using the diagonal method, we define the subsequence by choosing
$$
\{\phi_{\varepsilon_m}\}:=\{\phi_{\varepsilon^{(m)}_m}\},\ \varepsilon_m \in N_m.
$$
This diagonal subsequence converges at every rational point $q_k$. 
We denote the limit by:
$$
\varphi_0(q_k) = \lim_{m\to \infty}\phi_{\varepsilon_m}(q_k), \quad k=1,2,\dots
$$
Notice that pointwise limits preserve monotonicity, so $\varphi_0$ is increasing on $\Q\cap[0,2\pi)$.

\noindent
\paragraph{Step 2: The limit function behaves well at continuity points.}
We extend $\varphi_0$ to a function $\phi_0$ defined on the entire interval $[0,2\pi)$ by setting
$$
\phi_0(x) = \sup_{q_k \le x}\varphi_0(q_k). 
$$
Thus, $\phi_0$ is also increasing on $[0,2\pi)$. 
We claim that at every continuity point $x$ of $\phi_0$, 
$$
\lim_{m\to\infty}\phi_{\varepsilon_m}(x) = \phi_0(x). 
$$
Indeed, let $x \in [0,2\pi)$ be a point of continuity of $\phi_0$, for any $\delta>0$, there exist rational points $p, q \in \Q\cap[0,2\pi)$ s.t. $p < x < q$ and
$$
\phi_0(q) - \phi_0(p) < \delta. 
$$
By the monotonicity of $\phi_{\varepsilon_m}$, we have
$$
\phi_{\varepsilon_m}(p) \leq \phi_{\varepsilon_m}(x) \leq \phi_{\varepsilon_m}(q), \  \forall m\in \N. 
$$
Let $m\to\infty$, we obtain
$$
\varphi_0(p) = \liminf_{m\to\infty}\phi_{\varepsilon_m}(p) \leq \liminf_{m\to\infty}\phi_{\varepsilon_m}(x) \leq \limsup_{m\to\infty}\phi_{\varepsilon_m}(x) \leq \limsup_{m\to\infty}\phi_{\varepsilon_m}(q) = \varphi_0(q).
$$
Therefore, letting $\delta \to 0^+$, we conclude that the limit exists and
$$
\lim_{m\to\infty}\phi_{\varepsilon_m}(x) = \phi_0(x). 
$$

\noindent
\paragraph{Step 3: Dealing with potential discontinuity points.}
Since $\phi_0$ is a monotone function, its set of discontinuity points, denoted by $C = \{c_k\}_{k=1}^\infty$, is at most countable. 
We apply the diagonal argument (as in step 1) to the sequence $\{\phi_{\varepsilon_m}\}$ on the countable set $C$. 
This yields a further subsequence, denoted by $\{\phi_{\varepsilon_j}\}:=\{\phi_{\varepsilon_{j_m}}\}$, that converges pointwise on $C$. 

Hence, we redefine the values of $\phi_0$ on $C$ as
$$
\phi_0(y) := \lim_{j\to\infty}\phi_{\epsilon_j}(y), \quad y \in C. 
$$
This ensures that the subsequence $\{\phi_{\varepsilon_j}\}$ converges to $\phi_0$ pointwise everywhere on $[0,2\pi)$, and $\phi_0$ remains a monotone function.
\end{proof}

\begin{theorem} \label{theo:conv} 
Let $\gamma:S^1\to\R^2$ be a closed Lipschitz curve. 
Let $\gamma_\varepsilon\subset\R^4$ be the lifted Jordan curves defined in Section~\ref{sec:settings}, with $\gamma_0(t)=(\gamma(t),0,0)$. 
For each $0<\varepsilon\leq1$, let $u_\varepsilon:D^2\to\R^4$ denote a Douglas minimizer spanning $\gamma_\varepsilon$. 

Then there exists a sequence $\varepsilon_j\downarrow0$ and a map $u_0:\overline{D^2}\to\R^4$ belonging to $W^{1,2}(D^2)\cap C(\overline{D^2})$, such that 
$$
u_{\varepsilon_j}\rightharpoonup u_0 \ \text{weakly in } W^{1,2}(D^2), 
\ \text{strongly in } L^2(D^2),
$$
$$
u_{\varepsilon_j}\to u_0 \ \text{uniformly on }\forall K\subset\subset D^2. 
$$
Moreover, the boundary trace of $u_0$ satisfies
$$
u_0|_{\partial D^2}=\gamma_0\circ\phi_0,
$$
where $\phi_0:S^1\to S^1$ is a monotone parametrization.
\end{theorem} 

\begin{proof}
\textbf{Step 1: Uniform boundedness of $u_\varepsilon$ in $W^{1,2}(D^2)$.}

(i) \textit{Bounding the Dirichlet energy:} 
For any $\varepsilon \in (0,1)$, we construct a radial extension $v_\varepsilon(\rho,\alpha) = \rho\gamma_{\varepsilon}(\alpha)$ on the disk $D^2$, where the pair $(\rho,\alpha)$ is the polar coordinates. 
The Dirichlet energy of $v_\varepsilon$ is given by
\begin{align*}
    \mathscr{E}(v_\varepsilon) &= \frac{1}{2}\int_{D^2}|\nabla v_\varepsilon|^2 \, d\mathcal{H}^2 = \frac{1}{2}\int_0^{2\pi} d\alpha \int_0^1 \big(|\gamma_\varepsilon(\alpha)|^2 + |\gamma'_\varepsilon(\alpha)|^2\big) \rho \, d\rho \\[0.1cm]
    &\leq \pi \sup \big(|\gamma_\varepsilon|^2 + \mathrm{Lip}(\gamma_\varepsilon)^2\big).
\end{align*}
Since $\gamma$ is continuous on the compact set $S^1$, we can uniformly bound $\gamma_\varepsilon$ and its Lipschitz constant independent of $\varepsilon$ as following
$$
\|\gamma_\varepsilon\|_{L^\infty(S^1)} \leq \|\gamma\|_{L^\infty(S^1)} + 1 =: m_1, 
$$
$$
\text{Lip}(\gamma_\varepsilon) \leq \sqrt{\text{Lip}(\gamma)^2 + 2} =: m_2. 
$$
Because $u_\epsilon$ minimizes the area energy among all functions with the boundary condition $u_\varepsilon|_{\partial D^2} = v_\varepsilon|_{\partial D^2}$, correspondingly, it minimizes the Dirichlet energy (For more details, please see \cite{Pe2025}). 
Then, for all $\varepsilon>0$,
$$
\mathscr{E}(u_\epsilon) \leq \mathscr{E}(v_\varepsilon) \leq \pi(m_1^2 + m_2^2) =: \frac{1}{2}M_1^2,
$$
notice that $M_1$ depends only on $\gamma$. 
Consequently, the gradient $\|\nabla u_\varepsilon\|_{L^2(D^2)} \leq M_1$.

\vspace{0.2cm}

(ii) \textit{Bounding the $L^2$ norm:} Since $u_\varepsilon$ is harmonic, by the Weak Maximum Principle,
$$
\sup_{\overline{D^2}} |u_\varepsilon| \leq \sup_{\partial D^2} |u_\varepsilon| = \sup_{S^1} |\gamma_\varepsilon| \leq m_1. 
$$
It immediately follows that
$$
\|u_\varepsilon\|_{L^2(D^2)} \leq m_1 \left(\int_{D^2} 1 \, d\mathcal{H}^2\right)^{1/2} = \sqrt{\pi}m_1 =: M_2, 
$$
where $M_2$ depends only on $\gamma$.

Combining (i) and (ii), we deduce
$$
\|u_\varepsilon\|_{W^{1,2}(D^2)} = \|u_\varepsilon\|_{L^2(D^2)} + \|\nabla u_\varepsilon\|_{L^2(D^2)} \leq M_2 + M_1 =: M. 
$$
Therefore, there exists a constant $M > 0$, depending only on $\gamma$, such that $\|u_\varepsilon\|_{W^{1,2}(D^2)} \leq M$ for all $\varepsilon > 0$.

\noindent
\paragraph{Step 2: Extracting a convergent subsequence using compactness of $W^{1,2}(D^2)$.}
Since the family $\{u_\varepsilon\}$ is uniformly bounded in
$W^{1,2}(D^2)$, there exists a subsequence
$\{u_{\varepsilon_j}\}_{j=1}^\infty$ and a map
$u_0\in W^{1,2}(D^2)$ such that
$$
u_{\varepsilon_j}\rightharpoonup u_0
\qquad \text{weakly in } W^{1,2}(D^2).
$$
For simplicity, we continue to denote this subsequence by
$\{u_{\varepsilon_j}\}$.

By the Rellich-Kondrachov Theorem \cite{Ev2022}, the embedding $W^{1,2}(D^2) \hookrightarrow L^2(D^2)$ is strongly compact. 
Thus, passing to a further subsequence (still denoted by $\{u_{\varepsilon_j}\}$), we obtain strong convergence in $L^2(D^2)$
\[ u_{\varepsilon_j} \xrightarrow{j\to\infty} u_0 \quad \text{strongly in } L^2(D^2). \]

\noindent
\paragraph{Step 3: Uniform Convergence of $\{u_{\varepsilon_j}\}$.} 
We claim that $u_{\varepsilon_j} \to u_0$ uniformly on any compact subset $K \subset\subset D^2$.

To see this, we recall the interior Estimates for Derivatives of harmonic functions \cite{Ev2022}. 
For any compact set $K \subset\subset D^2$ and any multi-index $\alpha$ with $|\alpha| = m$, there exists a constant $C_m$, which depends on $m$ and $\text{dist}(K, \partial D^2)$, such that
$$
\sup_K |D^\alpha u_j| \leq C_m \|u_j\|_{L^1(D^2)}.
$$
Since $\{u_j\}$ is uniformly bounded in $W^{1,2}(D^2)$, H\"older's inequality implies that
$$
\|u_j\|_{L^1(D^2)} \leq \left( \int_{D^2} |u_j|^2 \, d\mathcal{H}^2 \right)^{1/2} \big( \mathcal{H}^2(D^2) \big)^{1/2} \leq \sqrt{\pi} \|u_j\|_{L^2(D^2)} \leq \sqrt{\pi} \|u_j\|_{W^{1,2}(D^2)} \leq \sqrt{\pi} M. 
$$
In particular, taking $|\alpha| = 1$, we have for any $x, y \in K$
$$
|u_j(x)-u_j(y)|\leq\sup_K |\nabla u_j| \cdot |x-y| \leq C_1 \sqrt{\pi}M|x-y|. 
$$
So the sequence $\{u_j\}$ is uniformly Lipschitz, and hence equicontinuous, on $K$.

Therefore, by the Arzel\`a-Ascoli theorem, there exists a further subsequence, still denoted by $\{u_{\epsilon_j}\}$, such that
$$
u_{\varepsilon_j} \xrightarrow{j\to\infty} u_0 \quad \text{uniformly on } K. 
$$

Applying the results above to higher-order derivatives, inductively, a similar argument shows that any derivative $D^\alpha u_j$ converges uniformly on $K$.

\noindent
\paragraph{Step 4: Fixing the boundary parameterization.}
By three-point normalization (see Section~\ref{sec:settings}), we may assume that the family $\{u_\varepsilon\}$ satisfies the three-point condition. 
That is, for three fixed distinct points $e^{i\theta_1}, e^{i\theta_2}, e^{i\theta_3} \in \partial D^2$ and three fixed distinct points $p_1, p_2, p_3 \in S^1$ such that $\gamma(p_1), \gamma(p_2), \gamma(p_3)$ are also distinct, we have
$$
u_\varepsilon(e^{i\theta_k}) = \gamma_\varepsilon(p_k), \ k = 1, 2, 3. 
$$
Recall that for Douglas's minimizer, the boundary values satisfy $u_\varepsilon|_{\partial D^2} = \gamma_\varepsilon \circ \Phi_\varepsilon$ for some monotonic mapping $\Phi_\varepsilon$. 
Applying the three-point normalization again, and noting that M\"obius transformations preserve monotonicity, we can find a monotone function $\phi_\epsilon: [0, 2\pi) \to [0, 2\pi)$ such that
$$
u_\varepsilon(e^{i\theta}) = \gamma_\varepsilon(\phi_\varepsilon(\theta)), \ \forall \theta \in [0, 2\pi). 
$$
Since the family of functions $\{\phi_\varepsilon\}$ is uniformly bounded (as mapping into $[0, 2\pi)$), Lemma~\ref{lem:mono} guarantees the existence of a subsequence, denoted by $\{\phi_{\varepsilon_j}\}_{j=1}^\infty$, that converges pointwise to some monotone function $\phi_0:[0,2\pi)\to[0,2\pi)$
\footnote{Although the range of $\phi_0$ should be $[0,2\pi]$ as a limit, however, the target curve $\gamma_\epsilon$ is a closed loop parameterized on $S^1 \cong \R/2\pi\mathbb{Z}$, so we canonically identify the value $2\pi$ with $0$, then restrict the range of $\phi_0$ to $[0, 2\pi)$ without changing the boundary values $u_0|_{\partial D^2}$.} with
$$
\phi_{\varepsilon_j}(\theta) \xrightarrow{j\to\infty} \phi_0(\theta) \ \ \text{pointwise on } [0, 2\pi). 
$$

\noindent
\paragraph{Step 5: Boundary convergence and continuous representative.}
We claim that the trace $u_0|_{\partial D^2}$ coincides with $\gamma_0 \circ \phi_0$, and $\gamma_0 \circ \phi_0$ is continuous.

Recall that the trace operator \cite{Ev2022} $\text{Tr}: W^{1,2}(D^2) \to H^{1/2}(\partial D^2)$ is a bounded linear operator. 
Since $u_{\varepsilon_j} \rightharpoonup u_0$ weakly in $W^{1,2}(D^2)$, it follows that
$$
u_{\varepsilon_j}|_{\partial D^2} = \text{Tr}(u_{\varepsilon_j}) \rightharpoonup \text{Tr}(u_0) = u_0|_{\partial D^2} \quad \text{weakly in } H^{1/2}(\partial D^2).
$$
By the compactness of the fractional Sobolev embedding $H^{1/2}(\partial D^2) \hookrightarrow L^2(\partial D^2)$ \cite{GiTr1977}, we can pass to a subsequence (still denoted by $\{u_{\varepsilon_j}\}$) such that $u_{\varepsilon_j}|_{\partial D^2} $ converges to $u_0|_{\partial D^2}$ strongly in $L^2(\partial D^2)$, i.e.
$$
\|u_{\varepsilon_j}|_{\partial D^2} - u_0|_{\partial D^2}\|_{L^2(\partial D^2)} \xrightarrow{j\to \infty} 0. 
$$
From this strong $L^2$ convergence, we can extract a further subsequence (again denoted by $\{u_{\epsilon_j}\}$) such that
$$
\lim_{j\to\infty} u_{\varepsilon_j}|_{\partial D^2}(e^{i\theta}) = u_0|_{\partial D^2}(e^{i\theta}) \quad \text{for}\  a.e.\   \theta \in [0, 2\pi).
$$
Also by our previous construction, $\gamma_{\varepsilon_j} \to \gamma_0$ uniformly. 
Together with $\phi_{\varepsilon_j} \to \phi_0$ pointwise, we have
$$
\gamma_{\varepsilon_j}(\phi_{\varepsilon_j}(\theta)) \xrightarrow{j\to\infty} \gamma_0(\phi_0(\theta)). 
$$
Letting $j\to \infty$, we deduce that the limit functions agree almost everywhere, i.e.
$$
u_0|_{\partial D^2}(e^{i\theta}) = \gamma_0(\phi_0(\theta)) \quad \text{for a.e. } \theta \in [0, 2\pi). 
$$
Because the trace $u_0|_{\partial D^2} \in H^{1/2}(\partial D^2)$, its Gagliardo seminorm \cite{GiTr1977} must be finite
$$
[\gamma_0 \circ \phi_0]_{H^{1/2}(\partial D^2)} = \left(\int_0^{2\pi} \int_0^{2\pi} \frac{|\gamma_0(\phi_0(\theta_1)) - \gamma_0(\phi_0(\theta_2))|^2}{|\theta_1 - \theta_2|^2} \, d\theta_1 d\theta_2\right)^{1/2} < \infty.
$$
Since $\phi_0$ is a monotonic function, it can only exhibit jump discontinuities. 
As $\gamma_0$ is a continuous mapping, if $\gamma_0 \circ \phi_0$ has discontinuities, they should be jumps as well.

Now assume $\gamma_0 \circ \phi_0$ has a jump discontinuity at some point, the numerator in the integrand above would be strictly bounded away from zero for $\theta_1, \theta_2$ on opposite sides of the jump. 
This would cause the double integral around the discontinuity to diverge, hence a contradiction. 
Thus, $\gamma_0 \circ \phi_0$ must be a continuous function on $[0, 2\pi)$.

\vspace{1.5mm}

Finally, note that the functions in $W^{1,2}(D^2)$ are defined up to a set of measure zero. 
Since $u_0|_{\partial D^2}$ equals the continuous function $\gamma_0 \circ \phi_0$ a.e., we can redefine $u_0|_{\partial D^2}$ on a measure-zero set so that they match exactly, i.e.
$$
u_0|_{\partial D^2} \equiv \gamma_0 \circ \phi_0 \quad \text{everywhere on } \partial D^2.
$$

\noindent
\paragraph{Step 6: Global continuity on the closed disk.}

We claim that $u_0 \in C(\overline{D^2})$. 
From the previous steps, we have established that $u_0 \in W^{1,2}(D^2)$ and is weakly harmonic in $D^2$; by Weyl's lemma \cite{GiTr1977}, it's harmonic;
Also, its trace $u_0|_{\partial D^2} = \gamma_0 \circ \phi_0$ is continuous on $\partial D^2$.

Using the Poisson integral formula\cite{Ev2022}, we could construct a classical harmonic function $v_0 \in C^{\infty}(D^2) \cap C(\overline{D^2})$ that continuously extends the boundary data $u_0|_{\partial D^2}$. 
More precisely, 
$$
v_0(x) = \frac{1-|x|^2}{2\pi}\int_{\partial D^2}\frac{u_0(y)}{|x-y|^2}\, d\mathcal{H}^1_y. 
$$

(i) Before comparing $u_0$ and $v_0$, we must ensure $v_0 \in W^{1,2}(D^2)$. Consider the admissible class of functions with the same boundary trace 
$$
\mathcal{A} := \{u \in W^{1,2}(D^2) :\  u|_{\partial D^2} = u_0|_{\partial D^2} \},
$$
since $u_0 \in \mathcal{A}$, the set is non-empty. 

\vspace{1.5mm}

The Dirichlet energy $\mathscr{E}(f) = {\frac{1}{2}\int_{D^2}|\nabla f|^2}$ is coercive and strictly convex on $\mathcal{A}$, guaranteeing the existence of a unique minimizer $\hat{u} \in \mathcal{A}$ which is weakly harmonic\cite{Ev2022}. 
By the uniqueness of solutions to the Dirichlet problem \cite{Ev2022}, the classical harmonic extension must coincide with the energy minimizer, i.e., $v_0 = \hat{u}$. 
Consequently, $v_0 \in W^{1,2}(D^2)$.

\vspace{2mm}

(ii) Now, we show that $u_0 = v_0$ almost everywhere. 
We define the difference function $$ w = u_0 - v_0. $$
Since $u_0, v_0 \in W^{1,2}(D^2)$ and they share the same trace on $\partial D^2$, we have $w \in W^{1,2}_0(D^2)$. 
Furthermore, because both $u_0$ and $v_0$ are harmonic in $D^2$, $w$ is also harmonic. Therefore, for any test function $\varphi \in W^{1,2}_0(D^2)$, it satisfies
$$
\int_{D^2} \nabla w \cdot \nabla \varphi \, d\mathcal{H}^2 = 0. 
$$
In particular, taking $w$ itself as the test function, we obtain
$$
\int_{D^2} |\nabla w|^2 \, d\mathcal{H}^2 = 0,
$$
which implies that $\nabla w = 0$ in $D^2$. 
Since $w$ has zero trace on the boundary, it follows that $w \equiv 0$ in $D^2$.

Thus, $u_0 \equiv v_0$, it's continuous on the closed disk $\overline{D^2}$, which concludes the proof.
\end{proof}

\begin{proposition} \label{prop:Minimality-of-the-Limit Surface}    
Let
$
u_{\varepsilon_j}:D^2\to\R^4
$
be the sequence of Douglas minimizers associated to the lifted curves
$\gamma_{\varepsilon_j}$, and suppose
$
u_{\varepsilon_j}\rightharpoonup u_0
$
weakly in
$
W^{1,2}(D^2).
$
Then $u_0$ minimizes area among all competitors spanning the limiting boundary curve $u_0|_{\partial D^2}$.
\end{proposition}

\begin{proof}

We will mostly transfer the proof from Proposition~\ref{prop:min-area} over here. 
Thanks to the weakly conformal property of $u_\varepsilon$, the area coincides with the Dirichlet energy. 
Since the Dirichlet energy functional in $W^{1,2}(D^2)$ is lower semi-continuous, so is the area functional. Thus, we can slightly adapt the comparison argument as follow.

With any Lipschitz competitor $v$ and its corresponding $v_\varepsilon$ as in Proposition \ref{prop:min-area}, by $u_{\varepsilon_j}\to u_0$ weakly in $W^{1,2}(D^2)$ and the lower semi-continuity of area functional,
\[
\mathrm{Area}(u_0)=\mathrm{Area}(\lim_{j\to\infty}u_{\varepsilon_j})\leq \liminf_{j\to\infty} \mathrm{Area}(u_{\varepsilon_j})\leq\liminf_{j\to\infty}\mathrm{Area}(v_{\varepsilon_j})= \mathrm{Area}(v).
\]
Therefore, $u_0$ is still minimizing the area.
\end{proof}
\begin{remark}
We actually have $\mathrm{Area}(u_0) = \displaystyle{\liminf_{j\to\infty}\mathrm{Area}(u_{\epsilon_j})}$ here. 
Suppose not. Then we may apply the same construction as in Proposition \ref{prop:min-area} to $u_0$ instead of $v$, obtaining a competitor $u_{0,\epsilon,\delta}$.

Note that $u_{\epsilon_j}\to u_0$ uniformly on each interior compact subset of $D^2$, and each $u_{\epsilon_j}$ is harmonic, it follows that $u_0$ is harmonic, hence is Lipschitz, on every interior compact subset of $D^2$.
In the above construction, we only use the restriction of $u_0$ to a compact subset contained in the interior of $D^2$, thus $u_{0,\epsilon,\delta}$ is an admissible  Lipschitz competitor. It shows that
\[\mathrm{Area}(u_0) < \liminf_{j\to\infty}\mathrm{Area}(u_{\epsilon_j})\le \liminf_{j\to\infty}\mathrm{Area}(u_{0,{\epsilon_j}}) \le \mathrm{Area}(u_0),\]
a contradiction.

\end{remark}
Notice that, even though $u_0$ is a mapping from $D^2 \to \R^4$, but its image actually lies in $\R^2$. If not, the consider the projection $P:\R^4 \to \R^2$ with $(r_1, r_2, r_3,r_4) \mapsto (r_1, r_2)$, then 
\[\mathrm{Area}(P\circ u_0) < \mathrm{Area} (u_0),\]
which is contradict to the minimality of $u_0$.

Therefore, we have the sets $P\circ u_0(D^2) = u_0(D^2)$ indeed.
\begin{remark}
    One interpretation of null homotopy is geometric: a continuous deformation of the geometric curve $\gamma(S^1)\subset \mathbb{R}^2$ into a point. In this sense, $u_0$ realizes such a deformation while minimizing the area among all Lipschitz competitors.

    The other interpretation is the parameterized one as in Definition \ref{def:nullhomo}, which requires preserving the boundary parameterization $\gamma:S^1\to\mathbb{R}^2$. To strictly follow this, we will construct a homotopy $\mathscr{h}_0:[0,1]\times S^1\to\mathbb{R}^2$ with boundary trace $\gamma$ in the next section.

\end{remark}

\subsection{Construction of the Minimal Homotopy} \label{subsec:constrct-minhom}
As we saw in Subsection~\ref{subsec:C1}, the map $u_0$ essentially gives a minimal null homotopy of $\gamma$, up to a boundary reparameterization from $\gamma\circ\phi_0$ to $\gamma$. 
It therefore remains to show that this reparameterization can be corrected without introducing additional swept area. 
In this section, we establish this uniformly for both the $C^1$ and Lipschitz settings. 

The main idea is, to use the first half of the time to sweeping out $u_0(D^2)$, while using the second half to reparameterize the boundary curve $\gamma_0(S)$. 
The key fact in making this construction well-defined is the implicit continuity of $\phi_0$.

Notice that, in both the $C^1$ and Lipschitz settings, we may assume that $\gamma$ is parameterized by arc length. In particular, by \cite{HeKoShTy2015},
$$
|\dot{\gamma}(t)|
=
\lim_{u\to t,\ u\neq t}
\frac{\mathrm{dist}(\gamma(t),\gamma(u))}{|t-u|}
= 1, \qquad \mathcal{H}^1\text{-a.e. } t.
$$
The same holds for $\gamma_0$.

We claim that $\phi_0$ is actually continuous. 
Indeed, if $\phi_0$ has a jumped discontinuity at $a$, i.e.
$$
\lim_{t\to a^-} \phi_0(t) = b_-<b_+=\lim_{t\to a^+} \phi_0(t).
$$
Since $\gamma_0\circ\phi_0$ is continuous, thus $\gamma_0(b_-)=\gamma_0(b_+).$ 
This implies that the parameter interval $[b_-, b_+]$ corresponds to a closed loop, whose length is $(b_+-b_-)>0.$ 

However, since $\gamma_0 \circ \phi_0(S^1) = u_0|_{\partial D^2}(S^1)=\gamma_0(S^1)$, both curves have the same image and therefore must have the same length, a contradiction.

Therefore, the minimal null homotopy of $\gamma_0$ would be
$$
\mathscr{h}_0(te^{i\theta})=\begin{cases}
u_0\big(2te^{i\theta}\big)\ \ & 0\leq t<1/2\\
\gamma_0\big((2-2t)\phi_0(\theta)+(2t-1)\theta\big) & 1/2\leq t\leq1.
\end{cases}
$$
Firstly, since $\phi_0$ is continuous, at the interface $t = 1/2$, $\mathscr{h}_0(te^{i\theta})$ is continuous. 
At the interface $t=1/2$, we have that 
$$
u_0(e^{i\theta}) = \gamma_0 \circ\phi_0(\theta) .
$$
At the outer boundary $t = 1$, we also have 
$$
\mathscr{h}_0(e^{i\theta}) = \gamma_0(\theta),
$$ 
which satisfies the required boundary parameterization. Therefore, $\mathscr{h}_0$ is a well-defined homotopy.

Secondly, the entire image of the outer annulus $A = \{te^{i\theta} :\  1/2 \le t \le 1\}$ is contained within $\gamma_0(S^1)$, which has zero area. 
This ensures that the reparameterization does not affect the minimality of area, i.e.
$$
\mathrm{Area}(\mathscr{h}_0)=\mathrm{Area}(u_0).
$$
As $u_0$ and $\gamma_0$ are actually in $\R^2$ (with 0 in both third and forth component), so $\mathscr{h}_0$ is the minimal homotopy we desired.

\section{Discussion}
We obtained minimal-area homotopies between immersed planar curves by lifting to higher codimension and passing to the limit of classical Plateau solutions. 
The construction is variational and remains entirely within the framework of maps and the area formula.

Several directions remain open. 
One is to better understand the geometry, regularity, and possible uniqueness of the limiting minimizer. 
Another natural direction is to study how the minimal swept area changes when the boundary curve is slightly perturbed. 
In particular, it would be interesting to understand whether small changes in the curve produce small changes in the minimum area and in the corresponding minimizing homotopies.
It would also be natural to investigate whether similar lifting constructions can be developed for higher-dimensional immersed manifolds. 
In particular, one may ask whether an immersed $n$-sphere in $\R^{n+1}$ can be lifted into higher codimension, approximated by embedded hypersurfaces, and studied through Plateau-type minimization and compactness arguments.
More generally, it would be interesting to understand whether variational approaches of this type can provide a broader framework for minimal homotopy problems beyond the planar setting.

\printbibliography
\end{document}